\documentclass[12pt, a4paper]{amsart}
\usepackage[hmargin=30mm, vmargin=25mm, includefoot, twoside]{geometry}
\usepackage[bookmarksopen=true]{hyperref}

\usepackage[alphabetic]{amsrefs}

\usepackage{amsfonts,amssymb,verbatim}
\usepackage{latexsym}
\usepackage{mathrsfs}
\usepackage{stmaryrd}
\usepackage{xspace}
\usepackage{enumerate, paralist}

\usepackage[usenames,dvipsnames]{color}
 \usepackage{txfonts, pxfonts}

\newtheorem{ThmIntro}{Theorem}

\newtheorem{PropIntro}[ThmIntro]{Proposition}

\newtheorem{thm}{Theorem}[section]
\newtheorem{cor}[thm]{Corollary}
\newtheorem{lem}[thm]{Lemma}

\newtheorem{prop}[thm]{Proposition}
\theoremstyle{definition}
\newtheorem{defn}[thm]{Definition}

\theoremstyle{remark}

\numberwithin{equation}{section}

\newcommand{\Z}{\mathbb{Z}}

\newcommand{\F}{\mathbb{F}}

\newcommand{\N}{\mathbb{N}}

\newcommand{\C}{\mathbb{C}}

\newcommand{\SL}{\text{SL}}

\newcommand{\HH}{\mathcal{H}}
\newcommand{\PP}{\mathcal{P}}

\newcommand{\eps}{\varepsilon}
\newcommand{\Lip}{\textnormal{Lip}}

\newcommand{\diam}{\textnormal{diam}}

\begin{document}
\title{Relative expanders}

\author{Goulnara Arzhantseva}
\address{Universit\"at Wien, Fakult\"at f\"ur Mathematik\\
Oskar-Morgenstern-Platz 1, 1090 Wien, Austria.}
\email{goulnara.arzhantseva@univie.ac.at}

\author{Romain Tessera}
\address{Laboratoire de Math\'ematiques, B\^atiment 425\\ Universit\'e Paris-Sud 11, 91405 Orsay, France
}
\email{romain.tessera@math.u-psud.fr}
\date{}

\thanks{The research of G.A.\ was partially supported by the ERC grant ANALYTIC no.\ 259527. The research of R.T.\ was partially supported by 
the ANR Blanc ANR-10-BLAN 0116, acronym GGAA} 
\baselineskip=16pt

\begin{abstract}
We exhibit a finitely generated  group $G$ and a sequence of finite index normal subgroups 
$N_n\trianglelefteqslant G$ such that  for every finite generating subset $S\subseteq G$, the sequence of finite Cayley graphs $(G/N_n, S)$ does not coarsely embed 
into any $L^p$-space for $1\leqslant p<\infty$ (moreover, into any uniformly curved Banach space), and
yet admits no weakly embedded expander. The reason why our examples do not coarsely embed is a new phenomenon called relative expansion, which we define in terms of Poincar\'e inequalities. \end{abstract}
\maketitle

\section{Introduction}
 A well-known obstruction for a metric space to coarsely embed into a Hilbert space is to admit a weakly embedded expander~\cites{M,Gr_sp,Gr_rw}. In this paper, we address 
 the following question:\smallskip

\emph{Given a metric space which does not embed coarsely into a Hilbert space, does it necessary contain a weakly embedded expander?}\smallskip

As a metric space weakly containing an expander does not coarsely embed into $\ell^p$ for all $1\leqslant p<\infty$, a counterexample is given by $\ell^p$ for any $2<p<\infty$ which does not coarsely embed into $\ell^2$ \cite{JR}  (see \cites{Ran,MN} for further results of the same kind).

However, the question remained open in the context of metric spaces with bounded geometry, and in particular for graphs with bounded degree (see \cite{O}*{Chapter 7}). We provide a strongly negative answer by constructing a sequence of $10$-regular finite graphs that does not coarsely embed into any $L^p$-space for any $1\leqslant p<\infty$, 
nor into any uniformly curved Banach space,
and yet does not admit any weakly embedded expander.

Before stating more precise results, we recall a few definitions.

\noindent{\bf Coarse embedding.}
Let $(X_n,d_n)_{n\in \N}$ be a sequence of metric spaces and let $(Y, d_Y)$ be a metric space.
A sequence of maps $\phi_n\colon X_n\to Y$ is a \emph{ coarse embedding of $(X_n)_{n\in \N}$ into $Y$} if there exist two 
proper functions $\rho, \gamma\colon [0,\infty)\to [0,\infty)$ such that for all $n\in \N$ and all $x,y\in X_n$,
$$\rho(d_n(x,y))\leqslant d_Y\left(\phi_n(x),\phi_n(y)\right)\leqslant \gamma(d_n(x,y)).$$

\noindent{\bf Expander.}
Given a finite connected graph $X$ with $\vert X\vert$ vertices and a subset $A\subseteq X$, denote by $\partial A$ the set of edges between $A$ and $X\setminus A$. 
The \emph{Cheeger constant} of $X$ is defined as $$h(X):=\min_{1\leqslant |A|\leqslant \vert X\vert/2}\frac{|\partial A|}{|A|}.$$ 

An \emph{expander} is a sequence $(X_n)_{n\in\mathbb{N}}$  of finite connected graphs with uniformly bounded degree, $\vert X_n\vert\to\infty$ as $n\to \infty$,  and
$h(X_n)\geqslant c$ uniformly over $n\in \mathbb{N}$ for some constant $c>0$.\smallskip

\noindent{\bf Weakly embedded sequence of finite metric spaces.}
This notion is used by Gromov \cite{Gr_rw} in his (random) construction of Gromov's monsters: finitely generated groups that do not coarsely embed into a Hilbert space. 

Let $(X_n)_{n\in \mathbb N}$ be a sequence of finite metric spaces  and let $Y$ be a metric space. A sequence of maps $\phi_n\colon X_n\to Y$ is a \emph{weak embedding} 
if there exists $K>0$ such that $\phi_n$ are $K$-Lipschitz, and for all $R>0$, 
 \begin{equation}\label{eq:weak}
 \lim_{n\to \infty}\sup_{x\in X_n}\frac{|\phi_n^{-1}(B(\phi_n(x),R))|}{|X_n|}=0, 
\end{equation}
where $B(y,R)$ denotes the ball of radius $R$ centered at $y\in Y$.

If $(X_n)_{n\in \mathbb N}$ satisfies that for all $R>0$, $\sup_{n, x \in X_n}|B(x,R)|<\infty$,  and $|X_n|\to \infty$ as $n\to\infty$, then a coarse embedding of 
the sequence $(X_n)_{n\in \mathbb N}$ into $Y$ is a weak embedding. In particular, a coarsely embedded expander is a weakly embedded expander. 
It is an open question whether the converse holds: does the existence of a weakly embedded expander imply the existence of a coarsely embedded expander?

On the other hand, if the target space $Y$ is such that for every $R>0$ we have $\sup_{y\in Y}|B(y,R)|<\infty$ (for instance, if $Y$ is the vertex set of a graph with bounded degree), then (\ref{eq:weak}) is equivalent to the following, see \cite{GK}*{Definition 5.4}:  $$\lim_{n\to \infty}\sup_{x\in X_n}\frac{|\phi_n^{-1}(\phi_n(x))|}{|X_n|}=0.$$ 
\smallskip

\noindent{\bf Box space.} 
Let $G$ be a finitely generated residually finite group and $S$ a finite generating subset.
The \emph{box space} of $G$, associated to $S$ and to a nested sequence of finite index normal subgroups
$G=N_0\trianglerighteqslant N_1 \trianglerighteqslant \cdots$ with trivial intersection 
$\bigcap_{n=0}^{\infty}N_n=\{ e \}$, is the sequence\footnote{Usually, the box space is defined to be a single metric space $X:=\sqcup^{\infty}_{n=1}X_n$, where the distance $d_{n,m}$ between $X_n$ and $X_m$, for $m\neq n$, is so that $d_{n,m}\to \infty$ as $\max\{n,m\}\to \infty$. One can easily check that a coarse (respectively, weak) embedding of such an individual box metric space $X$ 
provides  a coarse (respectively, weak) embedding of  the above sequence 
$(X_n, d_n)_{n\in \mathbb N}$, and vice versa.}  $(X_n)_{n\in \mathbb N}$ of Cayley graphs
$(X_n,d_n):= (G/N_n, d_{S_n}),$ where each $S_n$ is $S$ modulo $N_n$.

\

Our main theorem is the following. Recall that an $L^p$-space is a Banach space of the form $L^p(\Omega,\nu),$ where $(\Omega,\nu)$ is a measured space.

\begin{ThmIntro}\label{ThmIntro:MainExample}(Theorem \ref{thm:SL2})
There exist a finitely generated residually finite group $G$ and a box space $(Y_n)_{n\in \mathbb N}$ of $G$  
which does not coarsely embed into any $L^p$-space for $1\leqslant p<\infty$, neither into any uniformly curved Banach space, and yet does not admit any sequence of weakly embedded expanders.
\end{ThmIntro} 
Let us briefly outline our first example, the precise construction and its variants being given in~\S~\ref{sec:construc}. 

The rough idea is to use the fact that 
$\Z^2\rtimes \SL(2,\Z)$ has Kazhdan's Property~T relative to $\Z^2$, to deduce that any sequence of quotients $(Q_n)_{n\in \mathbb N}$, whose kernels have trivial intersection, does not coarsely embed into a Hilbert space. It is natural to start with the sequence $(\Z/n\Z)^2\rtimes \SL(2,\Z/n\Z)$, ${n\in \mathbb N}$. However, this sequence is not satisfactory as $\SL(2,\Z/n\Z)$, ${n\in \mathbb N}$, is well-known to be itself an expander. To outcome this major difficulty, our idea is to replace $\SL(2,\Z/n\Z)$ by an appropriately chosen finite extension of it. Namely,
we choose a suitable sequence of finite extensions that coarsely embeds into a Hilbert space, as provided by \cite{AGS}. This choice also 
guarantees that the resulting sequence admits no weakly embedded expander. 
This is about the main idea. In building our example, some technicalities appear, therefore, the actual details of the construction are more involved.

Theorem \ref{ThmIntro:MainExample} relies on the two main observations, of independent interest.

\subsection{Expanders and group extensions}
The following proposition roughly says that there are no expanders in an extension of two groups which admit coarse embeddings into a 
Hilbert space. It is an open question whether or not the coarse embeddability into a Hilbert space is preserved under taking extensions.

\begin{PropIntro}\label{PropIntro:exactsequence}
Let $(G_n)_{n\in\mathbb{N}}$ be a sequence of finitely generated groups equipped with finite generating sets $S_n$ of size $k$.
 We assume that for every $n$, there is an exact sequence 
$$1\to N_n\to G_n\to Q_n\to 1$$
such that:
\begin{itemize}
\item the sequence $(N_n)_{n\in\mathbb{N}}$ equipped with the induced metric coarsely embeds into a Hilbert space;

\item the sequence  $(Q_n)_{n\in\mathbb{N}}$ equipped with the word metric associated to the projection $T_n$ of $S_n$  coarsely embeds into a Hilbert space.  
\end{itemize}
Then, given a number $K>0$,  an expander $(X_n)_{n\in\mathbb{N}}$, and a sequence of $K$-Lipschitz maps $h_n\colon X_n\to Y_n=(G_n,S_n)$, there exist a constant $c>0$ and a sequence $y_n\in Y_n$ such that the cardinality of $h_n^{-1}(\{y_n\})$ is at least $c|X_n|$, and its diameter is $\geqslant c\; \diam(X_n)$.
\end{PropIntro}

\subsection{Relative Kazhdan's Property T and non-embeddability into $\ell^2$}
Let $G$ be a discrete countable group, and let $(\HH,\pi)$ be an orthogonal representation of $G$ in a Hilbert space $\HH$. A sequence $(v_n)_{n\in \mathbb N}$ of unit vectors in $\HH$ is called \emph{almost invariant}, if for all $g\in G$, $\|\pi(g)v_n-v_n\|\to 0$ as $n\to\infty$.

Let $Y\subseteq G$ be an infinite subset.
The pair $(G,Y)$ has \emph{relative Property~T}~\cite{C} if for  every almost invariant sequence  $(v_n)_{n\in \mathbb N}$, the convergence is uniform on $Y$, that is, $\sup_{y\in Y}\|\pi(y)v_n-v_n\|\to 0$  as $n\to\infty$.  

Equivalently~\cites{AW,C}, the pair $(G,Y)$ has relative Property T if for every affine isometric action of $G$ on a Hilbert space $\HH$, the orbits of $Y$ are bounded.
Every affine isometric action $\sigma$ of a group $G$ on a Hilbert space $\HH$ decomposes as $\sigma(g)v=\pi(g)v+b(g)$, where $\pi$ is a norm-preserving representation, and $b$ is a $1$-cocycle. An affine action is therefore characterized by the data $(\HH,\pi,b)$.

Somewhat opposite to relative Property T is the Haagerup Property: a discrete countable group has the \emph{Haagerup property} if it admits an affine isometric action on a Hilbert space $(\HH,\pi,b)$ such that $b$ is a coarse embedding (such an action is called \emph{metrically proper}).

The box space of a group without the Haagerup property does not coarsely embed into a Hilbert space by an observation of John Roe~\cite{R}. 
The following proposition strengthens this statement (perhaps, it is known among experts).
\begin{PropIntro}\label{PropIntro:RelativePoincare}
Let $G$ be a finitely generated group, let $S$ be a finite generating subset.
\begin{itemize}

\item[(i)] Assume $G$ has a an infinite subset $Y$ such that $(G,Y)$ has relative Property~T.  Then there exists $C>0$ such that for all finite quotients $Q$ of $G$, every function $f$ from $Q$ to a Hilbert space satisfies the following ``relative Poincar\'e" inequality: for every $y\in Y,$ 
\begin{equation}\label{eq:relativeT}
\sum_{g\in Q} \|f(g\bar{y})-f(g)\|^2 \leqslant C \sum_{g\in Q, s\in S} \|f(g\bar{s})-f(g)\|^2,
\end{equation}
where $\bar{y}$ and $\bar{s}$ denote the projection of $y$ and $s$ in $Q$. 
\item[(ii)] If one only assumes that $G$ does not have the Haagerup property, then there exist $C>0$ and  a sequence of probability measures $\mu_n$ whose support $W_n$ 
is finite and  $d(e,W_n)\to \infty$, such that 
for all $n\in \mathbb{N}$, every function $f$ from $Q$ to a Hilbert space satisfies:
\begin{equation}
\sum_{g\in Q}\left(\sum_{y\in W_n}\|f(g\bar{y})-f(g)\|^2\mu_n(y) \right) \leqslant C\sum_{g\in Q,s\in S} \|f(g\bar{s})-f(g)\|^2. 
\end{equation}
\end{itemize}
\end{PropIntro}

As an immediate consequence, we deduce 

\begin{cor}
Let $G$ be a finitely generated group and let $S$ be a finite generating subset. Assume that $G$ does not satisfy the Haagerup property. Then for any sequence of normal finite index subgroups 
$N_n\trianglelefteqslant G$ such that $\bigcap_{n=0}^{\infty}N_n=\{e\}$, the box space associated to $S$ and to $(N_n)_{n\in\mathbb{N}}$ does not coarsely embed into a Hilbert space. Moreover, if $G$ has relative Property T with respect to an infinite subset, then the box space does not coarsely embed into a Hilbert space. 
\end{cor}

\subsection{Non-embeddability  into $L^p$-spaces}
Observe that the results of the previous section do not provide any obstruction to embeddings into an $L^p$-space for $p> 2$ (note that coarse embeddability into $L^p$ are equivalent for all 
$1\leqslant p\leqslant 2$, see \cite{W}*{III.A.6} and \cite{S}). 
Using a different approach based on complex interpolation method due to V. Lafforgue, see \cite{P}*{\S 3}, we can nevertheless obtain -- under a certain condition -- a Poincar\'e inequality which is valid for all $L^p$-spaces, for $1\leqslant p<\infty$, and more generally for all uniformly curved Banach space, see \S \ref{sec:proofPropext} for the definition.

Let $G$ be a discrete countable group and let $A$ be a finite subset of $G$, denote by $M_A$ the averaging operator 
$$M_Af(g):=\frac{1}{|A|}\sum_{a\in A}f(ga),$$
for all $f\in \ell^2(G)$, and all $g\in G$. If $G$ is finitely generated,  and $S$ is a finite symmetric generating subset, then $M_S$ is the Markov operator associated to the simple random walk on the Cayley graph  $(G,S)$. On the other hand, if $H$ is a subgroup, then $M_H$ is the orthogonal projection on the subspace of $H$-right-invariant functions in $\ell^2(G)$.

\begin{PropIntro}\label{prop:poincareLp} Let $G$ be a finitely generated group, let $S$ be a finite symmetric generating subset.
Assume $G$ has a normal subgroup $H$ such that $(G,H)$ has relative Property~T.  Then, for every $1\leqslant  p<\infty$, there exist $C>0$ and $n_0\geqslant  1$, such that for every finite quotient $Q$ of $G$, every function $f$ from $Q$ to an $L^p$-space satisfies the following inequality:  
\begin{equation*}
\sum_{g\in Q} \|M_{q(H)}f(g)-f(g)\|^p \leqslant C\sum_{g\in Q} \|M_S^{n_0}f(g)-f(g)\|^p, 
\end{equation*}
where $q\colon G\twoheadrightarrow Q$ is the canonical projection. In particular, one deduces the following relative Poincar\'e inequality in $L^p$-spaces: for all $h\in H$,
\begin{equation*}
\sum_{g\in Q} \|f(g\bar{h})-f(g)\|^p \leqslant (2n_0)^pC\sum_{g\in Q,s\in S} \|f(g\bar{s})-f(g)\|^p, 
\end{equation*}
where $\bar{h}:=q(h)$ and $\bar{s}:=q(s)$.
\end{PropIntro}

\begin{PropIntro}\label{prop:unifcurved}
We keep the notation of Proposition \ref{prop:poincareLp}.
For every uniformly curved Banach space  $X$, there exist $C>0$ and $n_0\geqslant  1$, such that for every finite quotient $Q$ of $G$, every function $f$ from $Q$ to $X$ satisfies the following  inequality:  
\begin{equation*}
\sum_{g\in Q} \|M_{q(H)}f(g)-f(g)\|^2 \leqslant C\sum_{g\in Q} \|M_S^{n_0}f(g)-f(g)\|^2, 
\end{equation*}
In particular, for all $h\in H$ and $f$ from $Q$ to an $X$,
\begin{equation*}
\sum_{g\in Q} \|f(g\bar{h})-f(g)\|^2 \leqslant (2n_0)^2C\sum_{g\in Q,s\in S} \|f(g\bar{s})-f(g)\|^2. 
\end{equation*}
\end{PropIntro}

As a consequence we get the following non-embeddability result extending its well-known version for expanders.

\begin{cor}
Assume $G$ has an infinite normal subgroup $H$ such that $(G,H)$ has relative Property~T. Then for any sequence of normal finite index subgroups $N_n\trianglelefteqslant G$ such that $\bigcap_{n=0}^{\infty}N_n=\{e\}$, the box space associated to $S$ and to $(N_n)_{n\in\mathbb{N}}$  does not coarsely embed into any uniformly curved Banach space. 
\end{cor}

\subsection{Relative expanders and Poincar\'e inequalities}
 Given a sequence of groups $G_n$ equipped with left-invariant metrics, we say that a sequence of subsets $Y_n\subseteq G_n$ is {\it bounded} if there exists $R>0$ such that $Y_n$ lies in the ball of radius $R$ about the neutral element in $G_n$.

It is convenient to introduce the following terminology.
 Given a sequence of Cayley graphs $(Q_n,S_n)_{n\in\mathbb N}$ with $|S_n|\leqslant  k$, for some $k\in \N$, and a sequence of unbounded subsets $Y_n\subseteq Q_n$,  $(Q_n,S_n)_{n\in\mathbb N}$ is an \emph{expander relative to $Y_n$} if it satisfies (\ref{eq:relativeT}) with a constant $C>0$ independent of $n$. Clearly, the sequence $(G_n,S_n)_{n\in\mathbb N}$ is an expander if it is an expander relative to itself.  Unboundedness of $Y_n$ ensures that relative expanders do not coarsely embed into a Hilbert space. On the other hand, it is currently unclear whether this also prevents them from coarsely embedding into $L^p$-spaces for all $2\leqslant p<\infty$.

The phenomenon of relative expansion plays a crucial role\footnote{In \cite{ALW}, they use it to show that for certain sequences of finite Cayley graphs, expansion depends on the set of generators.} in \cite{ALW}, where it is used in order to show that certain finite semi-direct products $A_n\rtimes H_n$ form an expander. Indeed, in order to prove that their sequence is an expander, they need expansion to occurs both in $H_n$, and relative to $A_n$. 

Observe that this terminology is essentially group theoretic in nature. However,  the expansion relative to a sequence of {\it subgroups} has a nice graph-theoretic counterpart, namely expansion relative to a partition: given a sequence of finite connected graphs $(Y_n)_{n\in\mathbb N}$ with uniformly bounded degree,  and for every $n$ a partition $\PP_n$ of the vertex set of $Y_n$ into subsets such that $\max_{P\in \PP_n}|P|\to \infty$, say that the sequence $(Y_n)_{n\in\mathbb N}$ is an \emph{expander relative to $\PP_n$} if for every map $f\colon Y_n\to \HH$, 
  $$\sum_{P\in \PP_n}\sum_{x\in P} \|f(x)-M_P(f)\|^2 \leqslant C \sum_{x\sim y} \|f(x)-f(y)\|^2,$$
where $M_P(f):=\frac{1}{\vert P\vert}\sum_{x\in P}f(x)$, and where $x\sim y$ means they form an edge. In the case of a Cayley graph $(G_n,S_n)$, that is, of an expander relative to unbounded subgroups $H_n\leqslant G_n$, the elements of the partition $\PP_n$ are simply the right-cosets of $H_n$.

However, in the context of finite Cayley graphs (even of box spaces), we show that expansion relative to subgroups does not cover all possibilities:  in \S \ref{sec:wreath}, we construct a box space that does not coarsely embed into a Hilbert space and yet is not an expander relative to any sequence of subgroups. 
For those box spaces, it does not seem that one can obtain anything better than inequalities as  (\ref{eq:relativeT}). These inequalities can be interpreted as special cases of  Poincar\'e inequalities relative to a sequence of probability measures:  given a sequence of finite connected graphs $(Y_n)_{n\in\mathbb N}$ with uniformly bounded degree, a sequence $r_n\to \infty$ and for every $n$ a probability measure $\mu_n$ supported on $\Delta_{r_n}(Y_n)=\{(x,y)\in Y_n\times Y_n,\; d(x,y)\geqslant r_n\}$, say that the sequence $(Y_n)_{n\in\mathbb N}$ is an \emph{expander relative to $\mu_n$} if for every map $f\colon Y_n\to \HH$, $$\sum_{(x,y)\in \Delta_{r_n}(Y_n)} \|f(x)-f(y)\|^2 \mu_n(x,y)\leqslant \frac{C}{\vert Y_n\vert} \sum_{x\sim y} \|f(x)-f(y)\|^2.$$ 
For the sake of comparison, observe that an expander is an expander relative to the (renormalized) counting measure on $Y_n\times Y_n$. 
At the other extreme, recall \cites{T,O1}  that a metric space $Y$ does not coarsely embed into a Hilbert space if and only if there exist $r_n\to \infty$,  a sequence of probability measures $\mu_n$ supported on $\Delta_{r_n}(Y)$, and a constant $C<\infty$,  such that  for every $1$-Lipschitz map $f\colon Y\to \HH$, 
$$\sum_{(x,y)\in \Delta_{r_n}(Y)} \|f(x)-f(y)\|^2 \mu_n(x,y)\leqslant C.$$
Here is a slight improvement of this result that holds for sequences of finite Cayley graphs.

\begin{thm}
Let $(G_n,S_n)_{n\in\mathbb N}$ be a sequence of finite Cayley graphs with $|S_n|<\infty$, and let $1\leqslant p<\infty$. The following conditions are equivalent
\begin{itemize}
\item[(1)] The sequence  $(G_n,S_n)_{n\in\mathbb N}$ does not coarsely embed into $L^p$.

\item[(2)] Up to taking a subsequence,  there exist $r_n\to \infty$,  a sequence of probability measures $\mu_n$ supported on $\Delta_{r_n}(G_n)$, and a constant $C<\infty$,  such that  for every map $f\colon Y\to L^p$, 
\begin{equation}\label{eq:fortgeneralizedexp}
\sum_{g,g'\in G_n}\|f(g)-f(g')\|^p\mu_n(g,g')\leqslant \frac{C}{|G_n|}\sum_{g\sim g'} \|f(g)-f(g')\|^p.
\end{equation}
Moreover, $\mu_n$ is a  left-invariant measure on $G_n$:  $\mu_n(gg_1,gg_2)=\mu_n(g_1,g_2)$ for all $g,g_1,g_2\in G_n$.
\end{itemize}  
\end{thm}

\begin{proof}
Clearly (2) implies (1), so let us prove the converse implication. By the main result of \cite{T}, up to taking a subsequence,  there exist $r_n\to \infty$,  a sequence of probability measures $\nu_n$ supported on $\Delta_{r_n}(G_n)$, and a constant $C<\infty$,  such that  for every map $f\colon Y\to L^p$, 
\begin{equation}\label{eq:weakGeneralizedExp}
\sum_{g,g'\in G_n}\|f(g)-f(g')\|^p\nu_n(g,g')\leqslant C\Lip(f)^p,
\end{equation}
where $\Lip(f)$ is the Lipschitz norm of $f$.
We define $\mu_n$ as follows: $$\mu_n(g,g')=\frac{1}{|G_n|}\sum_{x\in G_n}\nu_n(xg,xg').$$
Then, given a function $f\colon G_n\to L^p$, we consider $\tilde{f}\colon G_n\to \ell^p(G_n,L^p)$ defined as $\tilde{f}(g)(h)=f(h g)-f(h)$ for all $g,h\in G_n$. Note that $$\Lip(\tilde{f})=\max_{s\in S_n}\left(\sum_{g\in G_n}\|f(g)-f(gs)\|^{2}\right)^{1/2}$$ Applying  (\ref {eq:weakGeneralizedExp}) to $\tilde{f}$ immediately yields (\ref {eq:fortgeneralizedexp}), ending the proof of the theorem.
\end{proof}

It is natural to ask whether our box space examples satisfy more ``expander like" Poincar\'e inequalities, i.e.\ where the sequence of measures $\mu_n$ can be replaced by the counting measure on some unbounded sequence of subsets of the form $A_n\times A_n\subseteq Y_n\times Y_n$. This is answered negatively in the Appendix, showing that the previous result is in some sense optimal. 

\subsection{Applications in K-theory}
The use of expanders in K-theory and operator algebra, especially, in the quest for counterexamples in higher index theory,
has been prevalent up to now. Certain classes of Margulis-type expanders (= the box spaces of Property T or Property~$\tau$ groups)  
are known to be counterexamples to (strong variants of) the Baum-Connes conjecture, 
although they do satisfy the coarse analogue of the Novikov conjecture and for some of them, even the maximal coarse Baum-Connes conjecture~\cites{HLS, WYu1, WYu2, OOY}. 

In \cite{CWYu}, the authors introduced the notion of \emph{fibered coarse embedding} into a Hilbert space~\cite{CWYu}*{Example 2.4}, which
is  a far generalization of a coarse embedding, sufficient for the maximal coarse Baum-Connes conjecture to hold. Instead of giving a precise definition here, let us only recall that a box space of a group $G$ admits a fibered coarse embedding if and only if $G$ has the Haagerup property~\cites{CWW,Finn}. 

 It is worthwhile to notice that by~\cite{Tu} (see also \cite{OOY}*{Corollary 4.18}) and \cite{MaNe}*{Section 10.5},  the maximal coarse Baum-Connes conjecture does hold for all the box spaces 
 that are constructed in \S \ref{sec:construc}. It follows, in particular, that the box spaces of Theorem~\ref{thm:SL2}  and Theorem~\ref{thm:wreath} provide interesting examples of metric spaces with bounded geometry satisfying the maximal coarse Baum-Connes conjecture and yet not admitting a fibered coarse embedding into a Hilbert space.
 Such examples can be obtained simply by taking any box space of $\Z^2\rtimes \SL(2,\Z)$, but once again, the originality of our examples comes from the fact that  non-\{fibered coarse embeddability\} of our box spaces is linked to relative expansion as opposed to expansion.

For the sake of completeness, in \S \ref{sec:fibered}, we slightly modify the construction of Theorem~\ref{thm:wreath}  in order to provide an example of a box space which does admit a fibered coarse embedding (and still does not coarsely embed into a Hilbert space, and does not weakly contain any expander).

\subsection{Organization}
In \S \ref{sec:prelim}, we recall some basic facts about expanders, coarse embeddings of abelian groups, and about the box space constructed in \cite{AGS}.
In \S \ref{sec:proofpropRP}, we prove Proposition \ref{PropIntro:RelativePoincare}, and in \S \ref{sec:proofPropext}, we prove Proposition \ref{sec:proofPropext}. In  \S \ref{sec:construc}, we 
prove our main result, Theorem \ref{ThmIntro:MainExample}. We also provide alternative constructions of box spaces which do not coarsely embed into a Hilbert space, and yet are not expanders relative to any sequence of subgroups, see Theorem \ref{thm:wreath}. Finally, in the last section we raise some open questions.

\subsection*{Acknowledgment} We are grateful to Masato Mimura for interesting suggestions. We thank Ana Khukhro, Mikhail Ostrovskii, Yves Stalder, and the anonymous referee for their remarks and corrections.

\section{Preliminaries}\label{sec:prelim}

\subsection{Expanders and Lipschitz embeddings}
Expanders have been pointed out by Gromov  as an obstruction for a metric space to coarsely embed into a Hilbert space~\cites{M,Gr_sp,Gr_rw}. 
This appeared in the context of his approach to the Novikov conjecture on the homotopy invariance of higher signatures.
We refer the reader to wonderful a monograph~\cite{Lub} and a survey~\cite{HLW} for an extensive information on expanders and their ubiquitous applications,
see also~\cite{NYu} for a recent account on their use in the context of the celebrated Baum-Connes and Novikov conjectures.

Let $1\leqslant p<\infty$.
Given a finite connected graph $X$, let $\beta_p(X)$ be the minimal $\beta>0$ such that
\begin{equation}\label{trueexpander}
    \frac{1}{|X|^2}\sum_{x,y\in X}\|f(x)-f(y)\|^p\leqslant \frac{\beta}{|X|}\sum_{x\sim y}\|f(x)-f(y)\|^p,
\end{equation}
for all maps  $f$ from  $X$ to an $L^p$-space.
It is easy to see (applying this inequality to the indicatrice function of a subset) that a sequence of finite connected graphs $(X_n)_{n\in \N}$ with uniformly bounded degree, 
$\vert X_n\vert\to\infty$ as $n\to \infty$, and such that $\sup_{n\in \N}\beta_p(X_n)>0$ is an expander. 
The converse is also true: for $p=2$, this is due to Alon \cite{A}, and the general case is due to Matou\v{s}ek~\cite{M}.

For convenience, we call an individual graph satisfying (\ref{trueexpander}) a  \emph{$(p,\beta)$-expander} graph.

\noindent{\bf Notation:} Given a set $X$ and a subset $A\subseteq X$, the complement of $A$ in $X$ is  denoted by $A^c$.  If $X$ is a metric space and $r\geqslant 0$, 
the \emph{$r$-neighborhood} of $A$ is denoted by $$[A]_r=\{x\in X, \; d(x,A)\leqslant r\}.$$

The following known fact is a special case of~\cite{BLMN}*{Lemma 5.4}. For convenience, we provide a short proof.

\begin{lem}\label{lem:subsexpander}
Let $1\leqslant p<\infty$, $\beta>0$ and $\alpha\in (0,1]$. Let $X$ be a $(p,\beta)$-expander graph of degree $\leqslant d$, and let $A$ be a subset of $X$ of  cardinality $\geqslant \alpha |X|$. 
Then there exist $\beta', c'>0$, only depending on $p,\beta,\alpha$ and $d$, such that $$\diam(A)\geqslant c'\diam(X),$$ and for all $1$-Lipschitz maps $f$ from  $A$ (equipped with the induced metric) to an $L^p$-space, we have 
\begin{equation}\label{subexpander}
    \frac{1}{|A|^2}\sum_{x,y\in A}\|f(x)-f(y)\|^p\leqslant \beta'.
\end{equation}
\end{lem}
\begin{proof}

The first statement is an easy exercise using that $h(X)$ is bounded below by a function of $\beta$. Let us show the second statement. 
Denote by $$h_{\alpha}=h_{\alpha}(X):=\min_{1\leqslant |\Omega|\leqslant (1-\alpha)\vert X\vert}\frac{|\partial \Omega|}{|\Omega|}.$$ 
One easily sees that $h_{\alpha}(X)$ is bounded below by some function of $h(X)$ and $\alpha$.
For all $i\in \N$, let $U_i:=[A]_i^c$. Observe that $U_i$ is decreasing and $\vert U_i\setminus U_{i+1}\vert\geqslant  \vert \partial U_i\vert/d \geqslant  (h_{\alpha}/d)\vert U_i\vert$, so that $$|U_i|\geqslant (1+h_{\alpha}/d)|U_{i+1}|.$$
This implies that $$|U_i|\leqslant \frac{|U_0|}{(1+h_\alpha/d)^i}\leqslant \frac{|X|}{(1+h_\alpha/d)^i}.$$
Let us denote $V_i:=U_i\setminus U_{i+1}$.

Now let us extend  $f$ to all of $X$ in the following way: for every point $x$ in $A^c$, choose a point $a_x$ in $A$ at minimal distance from $x$ and let $f(x):=f(a_x)$. Observe that if $y$ is a neighbour of $x$, then the distance from $a_x$ to $a_y$ is at most $d(x,A)+d(y,A)+1$. Now if $x$ belongs to $V_i$, then 
$d(a_x,a_y)\leqslant 2i+4.$  Since $f$ is $1$-Lipschitz, it follows that for all pairs of neighbours  $x$ and $y$ such that $x$ is in $V_i$, one has $\|f(x)-f(y)\|^p\leqslant 2^p$.

Moreover, for $x\in A$, one has $\|f(x)-f(y)\|^p\leqslant 1$. Therefore, 
\begin{eqnarray*}
\sum_{x\sim y}\|f(x)-f(y)\|^p & \leqslant & d|A|+ \sum_{i}\left(\sum_{x\in V_i, y\sim x} \|f(x)-f(y)\|^p\right)\\
                                     & \leqslant & d|X| +|X|\theta(h_{\alpha},d,p),
\end{eqnarray*}
where $\theta(h_{\alpha},d,p)=\sum_{i=0}^{\infty} (2i+4)^p/(1+h_{\alpha}/d)^i$, which is obviously a converging series.
We now apply (\ref{trueexpander}) to $f$:
\begin{eqnarray*}
 \frac{1}{|X|^2}\sum_{x,y\in X}\|f(x)-f(y)\|^p & \leqslant & \frac{\beta}{|X|}\sum_{x\sim y}\|f(x)-f(y)\|^p\\
                                                         &  \leqslant & \beta(d+\theta(h_{\alpha},d,p)).
\end{eqnarray*}
Since
 $$\frac{1}{|A|^2}\sum_{x,y\in A}\|f(x)-f(y)\|^p \leqslant  \frac{1}{\alpha^2|X|^2}\sum_{x,y\in X}\|f(x)-f(y)\|^p,$$
we deduce the lemma with $\beta':=\beta(d+\theta(h_{\alpha},d,p))/\alpha^2$.
\end{proof}
 
An immediate consequence of (\ref{subexpander}) is 
\begin{cor}\label{cor:proportion}
Under the assumptions of Lemma \ref{lem:subsexpander}, there is $x_0\in X$  such that at least $|A|/2$ points of $A$ are mapped  at distance at most $2\sqrt{\beta'}$ from $x_0$. Moreover, if $A=X$, then we can take $\beta'_p:=\beta d$.
\end{cor}

\subsection{A box space of the free group that coarsely embeds into a Hilbert space}

Our main result, Theorem \ref{ThmIntro:MainExample}, relies on a recent construction by Arzhantseva--Guentner--\v{S}pakula of a box space of the free group of rank $m\geqslant 2$, equipped with its standard generating subset, which coarsely embed into a Hilbert space~\cite{AGS}.

Let us start with some notation. Let $G$ be a group. For every $n\geqslant 0$, define inductively the characteristic subgroup  $\Gamma_n(G)$ of $G$ as $\Gamma_0(G):=G$, and $\Gamma_{n+1}(G)$ is the subgroup of $\Gamma_n(G)$ generated by squares of elements of $\Gamma_n(G)$. Observe that for all $l,k\geqslant 0$, $$\Gamma_{k+l}(G):=\Gamma_k(\Gamma_l(G)).$$
If $G$ is the free group $\F_m$ on $m\geqslant 2$ generators, then we denote by $H_n:= \F_m/ \Gamma_n(\F_m)$ the corresponding quotient. Denote by $T_n=\{t_1^{(n)},\ldots,t_m^{(n)}\}$ the image of the standard generating set of $\F_m$ in $H_n$.  
 
Our central tool is the following observation.
\begin{lem}\label{lem:factorization}
Let $H$ be a finite group of order $2^n$. Then $\Gamma_n(H)=\{e\}.$ In particular, if $H$ admits a set of generators $S=\{s_1,\ldots,s_m\}$, then there is an 
epimorphism $H_n\twoheadrightarrow H$ mapping each $t_i^{(n)}$ to  $s_i$.
\end{lem}
\begin{proof}
We prove this statement by induction on $n$. The case $n=0$ being trivial, we assume $n\geqslant 1$.
Being nilpotent, $H$ has a non-trivial abelianization. In an abelian $2$-group the set of squares is obviously a proper subgroup. It follows that $\Gamma_1(H)$ is a proper subgroup of $H$ and therefore has order $2^j$, with $j< n$. By induction hypothesis, we have $\Gamma_{j+1}(H)=\Gamma_j(\Gamma_1(H))=\{e\}$, and in particular $\Gamma_n(H)=\{e\}.$ 
\end{proof}

\begin{thm}\cite{AGS}\label{thm:ags}
The box space  of $\F_m$ associated to its standard set of generators and to the sequence $\Gamma_n(\F_m)$ coarsely embeds into a Hilbert space.
\end{thm}

In ~\cite{AGS}, the result is stated for $m=2$, however, all the details and proofs hold without modifications
for all $m\geqslant 2$.

\subsection{Coarse embeddings of abelian groups}

\begin{lem}\label{lem:abelian}
Let $\theta\colon[0,\infty)\to [0,\infty)$ be a proper function. 
Let $(A_n,d_n)_{n\in \mathbb N}$ be a sequence of abelian groups equipped with  invariant metrics $d_n$ such that for all $r$, balls of radius $r$  have cardinality at most $\theta(r)$. Then $(A_n,d_n)$'s are \emph{uniformly amenable}: for all $\eps>0$, and all $r>0$, there exist $D>0$ and a subset  $F_n\subseteq A_n$ of diameter $\leqslant D$ such that for all $g\in A_n$ such that $d_n(e,g)\leqslant r$, $$|gF_n\vartriangle F_n|\leqslant \eps |F_n|.$$  
\end{lem}
\begin{proof}
Let $A_n(r)$ be the subgroup of $A_n$ generated by the ball of radius $r$ with respect to $d_n$, that we shall denote by $B_{d_n}(e,r)$. Observe that if $g$ is at distance $\leqslant r$ from $e$, then multiplication by $g$ preserves $A_n(r)$. 

Let $d$ be the smallest integer larger than $\theta(r)$: the cardinality of the ball of radius $r$ in $A_n$ being $\leqslant d$, there exists a surjective homomorphism $\pi_n\colon\;\Z^d\twoheadrightarrow A_n(r)$, mapping the standard basis $\{v_1,\ldots,v_d\}$ surjectively to the ball of radius $r$ in $A_n$.  
In $\Z^d$, the volume of balls satisfies $C_d^{-1}R^d\leqslant B_{\Z^d}(e,R)\leqslant C_dR^d$ for some constant $C_d\geqslant 1$. It follows that there is a constant $K_d\in \N$ such that for all $R$, the ball of radius $2R$ in $\Z^d$ can be covered by at most $K_d$ balls of radius $R$. 

We pick some $R\in \N$ larger than $2K_d/\eps$. For every $k\in \N$, denote  $W_n(k):=\pi_n(B_{\Z^d}(e,k))$. We deduce from what precedes that $W_n(2R)$ can be covered by at most $K_d$ translates of $W_n(R)$. On the other hand, for every $g\in B_{d_n}(e,r)$, and every $k\geqslant 1$,  $$gW_n(k)\vartriangle W_n(k)\subseteq W_n(k+1)\setminus W_n(k-1).$$
Moreover, we have
$$\sum_{k=R}^{2R-1} |W_n(k+1)\setminus W_n(k-1)|\leqslant 2|W_n(2R)|.$$
By the pigeonhole argument, we deduce that there exists $R\leqslant k\leqslant 2R$, such that for all $g\in B_{d_n}(e,r)$,
$$|gW_n(k)\vartriangle W_n(k)|\leqslant  2|W_n(2R)|/R\leqslant  2K_d|W_n(k)|/R\leqslant \eps |W_n(k)|.$$
 It follows that $F_n:= W_n(k)$ satisfies $|gF_n\vartriangle F_n|\leqslant \eps |F_n|.$ 
On the other hand, the diameter of $F_n$, with respect to $d_n$, is less than $D:=2Rr$, so we are done.
\end{proof}

A coarse embedding of a sequence $(X_n, d_n)_{n\in\mathbb N}$ is often called a \emph{uniform coarse embedding} in contrast to  
a \emph{coarse embedding} $\phi\colon X\to E$ of an individual space $(X,d)$, viewed as a trivial sequence consisting of a single
space. An equivariant coarse embedding of a sequence then leads to a uniform variant of the Haagerup property.

A sequence of groups $(G_n,d_n)_{n\in \mathbb N}$ equipped with left-invariant distances satisfies the \emph{Haagerup Property uniformly}, if the $G_n$'s admit affine isometric actions $(\HH,\pi_n,b_n)$ such that $(b_n)_{n\in \mathbb N}$ is a uniform sequence of coarse embeddings, that is, with proper functions $\rho$ and $\gamma$ independent of $n$.
\begin{cor}\label{cor:abelian}
With the notation of Lemma \ref{lem:abelian}, the sequence of metric spaces $(A_n,d_n)$ satisfies the Haagerup property uniformly. In particular, such a sequence (uniformly) coarsely embeds into a Hilbert space. 
\end{cor}
\begin{proof}
Thanks to Lemma \ref{lem:abelian}, this simply follows from the standard proof that an amenable group satisfies the Haagerup property~\cites{AW,BCV}.
\end{proof}

\section{Proof of Proposition \ref{PropIntro:exactsequence}}

In this section we adopt the notation of Proposition \ref{PropIntro:exactsequence}.

Let $\phi_n\colon(N_n,d_{S_n})\to \HH$  and  $\psi_n\colon (Q_n,T_n)\to \HH$  be two coarse embeddings, i.e.\ such that there exist two proper functions $\rho, \gamma\colon [0,\infty)\to [0,\infty)$ such that for all $n\in \N$ and all $x,y\in N_n$,
\begin{equation}\label{eq:compressionphi}
\rho(d_{S_n}(x,y))\leqslant \|\phi_n(x)-\phi_n(y)\|\leqslant \gamma(d_{S_n}(x,y)),
\end{equation}
and similarly, for all $x,y\in Q_n$,
\begin{equation}\label{eq:compressionpsi}
\rho(d_{T_n}(x,y))\leqslant \|\psi_n(x)-\psi_n(y)\|\leqslant \gamma(d_{T_n}(x,y)).
\end{equation}
Using the argument of \cite{CTV}*{Section 3.4}, one can assume in addition that $\gamma(t)=t/K$ . Let us briefly recall the argument: by \cite{CTV}*{Lemma 3.11},  for every proper function $u\colon[1,\infty)\to [1/K^2,\infty)$, there exists a proper Bernstein function $F$ such that $F\leqslant u$. Pick one such $F$ such that $F(\gamma^2(t))\leqslant (t/K)^2$ for all $t\geqslant 1$. By Schoenberg's theorem \cite{S}, $(x,y)\to F(\|\phi_n(x)-\phi_n(y)\|^2)$ is a conditionally negative definite kernel on $N_n$, hence equal to $\|\phi'(x)-\phi'(y)\|^2$ for some map $\phi'\colon N_n\to \HH'$.  It follows by construction that $\phi'$ satisfies (\ref {eq:compressionphi}) with $\gamma(t)=t/K$.

Note that for $Q_n$, which is equipped with a word metric, this adjustment simply follows by rescaling the Hilbert space norm.

Let $\beta>0$, $d$ be such that for all $n$, $X_n$ is an $(2,\beta)$-expander of degree $\leqslant d$.

Now, let $h_n\colon X_n\to G_n$ be $K$-Lipschitz maps and consider the composition $f_n:=\psi_n\circ \pi_n\circ h_n$, where $\pi_n$ is the projection from $G_n$ to $Q_n$. By construction, $f_n$ is a sequence of $1$-Lipschitz maps from $X_n$ to $\HH$. We apply Corollary~\ref{cor:proportion}, with $A:=X_n$ and deduce that there exist $x_n\in X_n$ and a subset $A_n\subseteq X_n$ of cardinality at least $|X_n|/2$ such that $f_n(A_n)$  is mapped inside the ball of radius $2\sqrt{\beta d}$ around $f_n(x_n)$. By properness of $\rho$, (\ref{eq:compressionpsi}) implies that $\pi_n\circ h_n(A_n)$ is mapped inside the ball $B(\pi_n\circ h_n(x_n), r)$, where $r:=\rho^{-1}(2\sqrt{\beta d})$. As this ball has cardinality at most $k^r$,  there exists a subset $A_n'$ of $A_n$ of size at least $ \vert A_n\vert/k^r\geqslant \vert X_n\vert/(2k^r)$ which is mapped by $\pi_n\circ h_n$ to a single point in $Q_n$. Therefore, $h_n(A_n')$ is mapped inside a coset of $N_n$. Up to composing $h_n$ by a left translation in $G_n$, we can assume that $h_n(A_n')$ is mapped into $N_n$. 

We now apply Corollary~\ref{cor:proportion} to the map $f'_n:=\phi_n\circ h_n\colon\;A_n'\to \HH$, this yields $\beta'>0$, and sequences $y_n\in A_n'$ and $A_n''\subseteq A_n'$ satisfying $|A_n''|\geqslant |A_n'|/2$, such that $f'_n(A_n'')\subseteq B(f'_n(y_n),2\sqrt{\beta'})$. We deduce that $h_n(A_n'')\subseteq B(y_n,r')$, where $r':=\rho^{-1}(2\sqrt{\beta'})$. Hence, there exists a subset of $A_n''$ of size at least $|A_n''|/k^{r'}\geqslant |X_n|/(4k^{r+r'})$ which is mapped by $h_n$ to a single point. This ends the proof of 
Proposition \ref{PropIntro:exactsequence}.

\section{Proof of Proposition \ref{PropIntro:RelativePoincare}}\label{sec:proofpropRP}

Let us start with an easy lemma.

\begin{lem}\label{lem:relativeT}
If the pair $(G,Y)$ has relative Property $T$, then for every finite generating subset $S$ of $G$, there exists a constant $C>0$ such that for every affine isometric action $(\HH,\pi,b)$ of $G$ one has
$\|b(y)\|^2\leqslant C\sum_{s\in S}\|b(s)\|^2$ for all $y\in Y$. 
\end{lem}
\begin{proof}
Assume on the contrary that there exists a sequence $(\HH_n,\pi_n,b_n)_{n\in\mathbb N}$ of isometric actions such that $\sum_{s\in S}\|b_n(s)\|^2\leqslant 1$ for all $s\in S$, and a sequence $y_n\in Y$ such that $\|b_n(y_n)\|\to \infty$ as $n\to\infty$. Up to passing to a subsequence, we can assume that $\|b_n(y_n)\|^2\geqslant 2^n$. Then consider the norm-preserving representation of $G$ obtained by taking the direct sum of $\bigoplus_n(\pi_n,\HH_n)$ and the cocycle defined by $b:=\sum_nb_n/n$. As $\|b(s)\|^2\leqslant \sum_n1/n^2$, $b$ is well defined. On the other hand, $\|b(y_n)\|^2\geqslant 2^n/n^2$, hence $b$ is not bounded on $Y$, a contradiction. 
\end{proof}

In this section, we adopt the notation of Proposition \ref{PropIntro:RelativePoincare}. We start with a map 
$f\colon Q\to \HH$. Consider the Hilbert space $\HH':=\ell^2(Q,\HH)$, which is a convenient way of denoting the orthogonal direct sum of copies of $\HH$ indexed by $Q$. Denote by $|\cdot|$ the norm on $\HH$ and by $\|\cdot\|$ the norm on $\HH'$. Let $\pi$ be the (norm-preserving) action of $G$ by right-translation and consider the affine action $\sigma$ of $G$ on $\HH'$ associated to the $1$-cocycle 
$$b(x):=f-\pi(x)f.$$ We have
$$\|b(x)\|^2=\sum_{g\in Q}|f(g)-f(gx)|^2.$$
Now using Lemma~\ref{lem:relativeT}, we deduce that $\|b(y)\|^2\leqslant C\sum_{s\in S}\|b(s)\|^2$ for some constant $C>0$ only depending on $G$ and $S$. We  precisely recover (\ref{eq:relativeT}). 

The proof of the second statement relies on the same argument, but using the following result instead of Lemma \ref{lem:relativeT}
\begin{thm}\cite{T}*{Theorem 7}
A  group  $G$ generated by a finite subset $S$, does not have the Haagerup Property if and
only if  there exist a number $C>0$ and a sequence of (symmetric) probability measures $\mu_n$ with finite support $W_n$ satisfying $d(e,W_n)\to \infty$,  such that every affine isometric action $(\HH,\pi,b)$ of $G$ satisfies
$\sum_{y\in W_n}\|b(y)\|^2\mu_n(y)\leqslant C\sum_{s\in S}\|b(s)\|^2.$
\end{thm}

\section{Proof of Proposition \ref{prop:poincareLp}}\label{sec:proofPropext}
Let $1\leqslant p\leqslant \infty$.
Let $(\Omega,\mu)$ be a measured space and denote $\ell^p(Q,L^p)$ the vector space of functions from $Q$ to $L^p(\Omega,\mu)$ equipped with the $L^p$-norm 
$$\|f\|_{\ell^p}:=\sum_{g\in Q}\|f(g)\|_p^p,$$
for all $f\colon Q\to L^p$. 

The operators $M_S$ and $M_{q(H)}$ extend naturally to operators on $\ell^p(Q,L^p)$. Moreover,  since $H$ is normal, $M_S$ and $M_{q(H)}$ commute with each other. 

Since the pair $(G,H)$ has relative Property T, we can apply (\ref{eq:relativeT}) for $Y=H$: for all functions $f\in \ell^2(Q,L^2)$, this equation says that 
$$\sum_{g\in Q}\|(f-M_{q(H)}f)(g)\|^2\leqslant C \sum_{g\in Q, s\in S} \|f(g\bar{s})-f(g)\|^2.$$
An easy computation shows that 
$$ \sum_{g\in Q, s\in S} \|f(g\bar{s})-f(g)\|^2=2|S|\sum_{g\in Q}\langle (1-M_S)f(g),f(g)\rangle_{\ell^2}.$$ Therefore, we have
$$\|(1-M_{q(H)})f\|_{\ell^2}^2\leqslant 2C|S|\langle (1-M_S)f,f\rangle_{\ell^2},$$
which implies that, while restricted to the orthogonal subspace to $H$-invariant vectors, the operator $M_S$ has norm $<1$. Since $1-M_{q(H)}$ is the orthogonal projector onto this subspace, and since this projector commutes with $M_S$, this implies that the operator $M_S(1-M_{q(H)})$ has norm $<1$.
On the other hand, observe that $M_S(1-M_{q(H)})$ has norm $\leqslant 2$ for its actions on both $\ell^{\infty}(Q,L^{\infty})$ and  $\ell^{1}(Q,L^{1})$. Hence we obtain by interpolation  (see \cite{BL}) that for $n_0$ large enough (depending on $p$), the operator norm of $(M_S(1-M_{q(H)}))^{n_0}=M_S^{n_0}(1-M_{q(H)})$ for its action on $\ell^p(Q,L^{p})$ is $\leqslant 1-c_p$, for some $c_p>0$. It follows that  
$$\|h\|_{\ell^p}\leqslant  \frac{1}{c_p}\|h-M_S^{n_0}(1-M_{q(H)})h\|_{\ell^p},$$
for all $h\in \ell^p(Q,L^p)$.
Using that $1-M_{q(H)}$ is an involution, and applying the previous inequality to $h:=(1-M_{q(H)})f$, where $f\in \ell^p(Q,L^{p})$, we obtain
\begin{eqnarray*}
\|(1-M_{q(H)})f\|_{\ell^p} & \leqslant &  \frac{1}{c_p} \|(1-M_{q(H)})f-M_S^{n_0}(1-M_{q(H)})f\|_{\ell^p}\\
                              & = &  \frac{1}{c_p} \|(1-M_S^{n_0})(1-M_{q(H)})f\|_{\ell^p}\\
                           &  =  & \frac{1}{c_p}  \|(1-M_{q(H)})(1-M_S^{n_0})f\|_{\ell^p} \\
                           & \leqslant & \frac{2}{c_p}\|(1-M_S^{n_0})f\|_{\ell^p}.
\end{eqnarray*}
Then Proposition \ref{prop:poincareLp} follows, taking $C:=(2/c_p)^p$.

\section{Proof of Proposition \ref{prop:unifcurved}}

Given a measure space $(\Omega,\mu)$ and a Banach space $X$, denote by $L^2(\Omega,\mu,X)$ the (Banach) space of square summable measurable functions from $\Omega$ to $X$, equipped with the obvious norm.  The following class of Banach spaces was pointed out by V. Lafforgue who asked for a more concrete characterization: 
\begin{defn}\cite{P}*{\S 2}
A Banach space $X$ is \emph{uniformly curved} if it satisfies the following property: there is a function $\eps\to\Delta_X(\eps)$ tending to zero with $\eps > 0$ such that for all measured space $(\Omega,\mu)$, every operator $T\colon L^2(\Omega,\mu)\to L^2(\Omega,\mu)$ with $\|T\|\leqslant \eps$ that is simultaneously of norm $\leqslant 1$ on $L^1$ and on $L^{\infty}$ must be of norm 
$\leqslant \Delta(\eps)$ on $L^2(\Omega,\mu,X)$.
\end{defn}\label{def:unifcurved}
Observe that the proof of Proposition \ref{prop:poincareLp} immediately extends (for $p=2$) to Proposition \ref{prop:unifcurved}. For this reason, we shall not repeat the argument. More precisely, Definition \ref{def:unifcurved} replaces the interpolation argument that was used there, and provides in a sense the minimal requirement for this proof to work.

 Let us finally mention a remarkable fact proved by Pisier, showing that uniformly curved spaces are intimately related to complex  interpolation:  the subclass of uniformly curved spaces whose associated $\Delta_X$ can be taken of the form $\Delta_X(\eps)=\eps^{\delta}$ with $\delta>0$ coincides with the  
 a priori much smaller class of subquotients of $\theta$-Hilbert spaces (these are Banach spaces obtained by complex interpolation between a family of Banach spaces and a Hilbert space). He also provides a (more technical) characterization of uniformly curved Banach spaces \cite{P}*{Theorem 9.2}. 
 
\section{Main result: proof of Theorem \ref{ThmIntro:MainExample}, and other constructions}\label{sec:construc}

Let us start with a convenient notation. Let $V$ be a group, let $H$ be a group of automorphisms of $V$, and let $K$ be a group surjecting to $H$. Denote by $V\rtimes_H K$ the semi-direct product, where $K$ acts on $V$ via its quotient $H.$

Let $Q$ be the kernel of the morphism $\SL(2,\Z)\twoheadrightarrow\SL(2,\Z/2\Z)$. This group is   generated by 3 elements\footnote{Generators are
$ \left( \begin{array}{cc}
-1 & 0 \\
0 & -1 \\
\end{array} \right)$,
$ \left( \begin{array}{cc}
1 & 2 \\
0 & 1 \\
\end{array} \right)$,
$ \left( \begin{array}{cc}
1 & 0 \\
2 & 1 \\
\end{array} \right)$.}. 
Let $\pi\colon\F_{3}\twoheadrightarrow Q$ be the morphism mapping the standard generating subset $U:=\{u_1,u_2,u_{3}\}$ of $\F_{3}$ to some generating subset of $Q$.

We shall construct a box space for the following group 
$$G:=\Z^2\rtimes_{Q} \F_3,$$ 
where $Q$ acts in the standard way on $\Z^2$. A generating subset $S$ of cardinality $5$ for $G$ is the union of $U$ with standard basis of $\Z^2$.

Now, for all $n\geqslant 1$, the image $Q_n$ of $Q$ in $\SL(2,\Z/2^n\Z)$ has order $2^{3n-3}.$ Let $U_n:=\{u_1^{(n)}, u_2^{(n)}, u_{3}^{(n)}\}$ be the generating subset of $Q_n$ obtained by projecting $U$. Consider the projection $H_{3n-3 }\twoheadrightarrow Q_n$ of  Lemma \ref{lem:factorization} (with $m=3$) mapping $T_{3n-3}$ to $U_n$, and let 
$$G_n:=(\Z/2^n\Z)^2\rtimes_{Q_n} H_{3n-3}.$$
Let $K_n$ be the kernel of the surjective morphism  $p_n\colon G\twoheadrightarrow G_n$.  Observe that, setting $G_0:=\{e\}$ and $K_0:=G$, we have $\cap_{n=0}^{\infty}K_n=\{e\}$.
Therefore, $\sqcup_{n=0}^{\infty}(G_n,S_n)$, where $S_n$ is the projection of $S$, is a box space of $G$ associated to $S$ and to a nested sequence of finite index normal subgroups
$G=K_0\trianglerighteqslant K_1 \trianglerighteqslant \cdots$ with trivial intersection.

Our main result, Theorem \ref{ThmIntro:MainExample}, follows from the following more precise result.
Let $1\leqslant p<\infty$.
\begin{thm}\label{thm:SL2}
The sequence of Cayley graphs $(G_n,S_n)_{n\in\mathbb{N}}$ satisfies the two following properties
\begin{itemize}
\item There exists $C>0$ such that, for all $n\in \N$,  every function  $f$ from $G_n$ to an $L^p$-space satisfies: for every $h\in (\Z/2^n\Z)^2,$
$$\sum_{g\in G_n} \|f(gh)-f(g)\|^p \leqslant C \sum_{g\in G_n, s\in S_n} \|f(gs)-f(g)\|^p.$$ If $X$ is a uniformly curved Banach space, then this inequality holds for functions with values in $X$, and for some  $1\leqslant p<\infty$.
In particular, the sequence $(G_n,S_n)_{n\in \N}$ does not uniformly coarsely embed into any uniformly curved Banach space. 
\item On the other hand, if $(X_n)_{n\in\mathbb N}$ is a sequence of expander graphs and if $h_n\colon X_n\to (G_n,S_n)$ is a sequence of $1$-Lipschitz maps, there exist a constant $c>0$ and a sequence $y_n\in (G_n, S_n)$ such that the cardinality of $h_n^{-1}(\{y_n\})$ is at least $c|X_n|$ and its diameter is $\geqslant c\; \diam(X_n)$.  
\end{itemize}

\end{thm}
\begin{proof}
First, note that the pair $(G,\Z^2)$ has relative Property $T$: indeed, this follows from the proof of \cite{Burger}*{Proposition 1}.
Therefore, the first statement follows from Propositions \ref{PropIntro:RelativePoincare} and \ref {prop:unifcurved}. 

For every $r\in \N$, the ball of radius $r$ in $(G_n,S_n)$ has cardinality $\leqslant |S|^r$. Hence by Lemma \ref{lem:abelian} and Corollary \ref{cor:abelian}, the sequence of subgroups $(\Z/2^n\Z)^2$ of $G_n$, equipped with the induced distance $d_{S_n}$, uniformly coarsely embeds into a Hilbert space $\HH$.  
On the other hand, Theorem~\ref{thm:ags} provides such a coarse embedding into $\HH$ for the sequence $(H_{3n-3}, T_{3n-3})_{n\in\mathbb{N}}$. 
The second statement of the theorem therefore results from Proposition \ref{PropIntro:exactsequence}.
\end{proof}

In our construction, we use the fact that $Q_n$ is a 2-group in order to apply Lemma~\ref{lem:factorization}. 
One can also take $Q_n$ to be an $l$-group for any fixed integer $l>2$ and coarsely embeddable box spaces from~\cite{Kh} which generalize those of \cite{AGS}.
An analogue of Lemma~\ref{lem:factorization} holds in this case as well, and is proved in the same way.

\subsection{A construction based on generalized wreath products}\label{sec:wreath}
Using a result of \cite{CI}, we produce examples with expansion relative to an unbounded sequence of subsets, but not relative to any (unbounded) sequence of subgroups.
Our construction is based on generalized wreath products.

Given two countable discrete groups $A, G$, and a countable set $Z$ equipped with a $G$-action, the generalized wreath product $A\wr_Z G$ is the 
semi-direct product $\bigoplus_{Z}A\rtimes G.$
In \cite{CI}, it is proved that if $Q$ is a quotient of $G$ with Property $T$, then for every $Q$-orbit $W$ in $\bigoplus_{Z}A$, the pair $(A\wr_Z G,W)$ has relative Property T. More generally, if $Q$ does not have Haagerup property, then neither does $A\wr_Q G$.

Let us construct more examples of relative expanders which admit no any weakly embedded expanders.

We consider $Q:=\ker(\SL(3,\Z)\twoheadrightarrow \SL(3,\Z/2\Z))$, equipped with a finite generating subset $U$. 
 Our new example will be a box space of the group $G:=\Z\wr_{Q}\F_U$. For the moment, let us keep an arbitrary finitely generated group $A$ as the ``lamp group". 

Let $V$ be a generating subset of $A$.  
A generating subset of $G$ is then given by $S:=\{v_e, \; v\in V\}\sqcup U$, where the notation $v_g$ for $g\in Q$ means the element of $\bigoplus_{Q} A$ which equals $v$ at $g$ and $0$ everywhere else.
 
We consider the sequence of Cayley graphs $(G_n,S_n):=(A_n\wr_{Q_n}H_{8n-8},S_n),$ where 
\begin{itemize}
\item $Q_n$ is the image  of $Q$ in $\SL(3,\Z/2^n\Z)$, observe that $Q_n$ is a 2-group of order $2^{8n-8}$; 
\item $H_{8n-8}$ is given by Lemma \ref{lem:factorization};
\item $A_n$ is a finite quotient of $A$, and $V_n$ is the corresponding projection of $V$;
\item $S_n:=\{v_e, \; v\in V_n\}\sqcup U_n$, where $U_n$ is the projection of $U$ under the epimorphism $\F_U\twoheadrightarrow H_{8n-8}$.
\end{itemize}

Observe that $(G_n,S_n)$ is a box space of $(G,S)$, provided $(A_n,V_n)$ is a box space of $(A,V)$. For every $n$, consider the subset $$Y_n:=\{v_g, \; g\in Q_n, v\in V_n\}.$$ The sequence $Y_n$ is clearly unbounded, it is not even of bounded size.

Let us start with an easy but useful lemma.

\begin{lem}\label{lem:extension}
Let $H=N\rtimes R$ be a semi-direct product of groups, and let $T$ be a subset of $N$, invariant under conjugation by $R$. Then, the action of $N$ on its Cayley graph $(N,T)$ extends to an action of $H$ (by graph isomorphisms).
\end{lem}
\begin{proof}
Denote the action of $R$ on $N$ by  $r\cdot n$, where $r\in R$ and $n\in N$.  
Therefore the product in the semi-direct product reads $(n_1,r_1)(n_2,r_2)=(n_1(r_1\cdot n_2),r_1r_2)$.
Define the action of $H$ on $N$ by $(n,r)\circ n_0:=n(r\cdot n_0).$ To see that this is a well-defined action, let us compute
\begin{eqnarray*}
((n_1,r_1)(n_2,r_2))\circ n_0 & = & n_1(r_1\cdot n_2)((r_1r_2)\cdot n_0) \\
                                        & = & n_1(r_1\cdot n_2)(r_1\cdot(r_2\cdot n_0)) \\
                                        & = & n_1(r_1\cdot(n_2(r_2\cdot n_0)))\\
                                        & = & n_1(r_1\cdot ((n_2,r_2)\circ n_0))\\
                                        & = & (n_1,r_1)\circ((n_2,r_2)\circ n_0).
\end{eqnarray*}
Moreover, for all $n_0,n,t\in N$ and $r\in R$, we have $$(n,r)\circ (n_0t)=n(r\cdot n_0)(r\cdot t).$$ Therefore, since $T$ is invariant under conjugation by $R$, this action also preserves the graph structure.
\end{proof}

We now provide examples of box spaces that do not coarsely embed into a Hilbert space but yet are not  expanders relative to any unbounded sequence of subgroups.  Note that we do not know whether the following sequence of graphs coarsely embeds 
into an $L^p$-space for large $p$.
\begin{thm}\label{thm:wreath}
The sequence of Cayley graphs $(G_n,S_n)_{n\in\mathbb{N}}$ satisfies the following properties.
\begin{itemize}
\item There exists $C>0$ such that, for all $n\in \N$,  every function  $f$ from $G_n$ to a Hilbert space satisfies: for all $y\in Y_n$,
\begin{equation}\label{eq:PoincareSubset}
\sum_{g\in G_n} \|f(gy)-f(g)\|^2 \leqslant C \sum_{g\in G_n, s\in S_n} \|f(gs)-f(g)\|^2.
\end{equation}
In particular, the sequence $(G_n,S_n)_{n\in \N}$ does not uniformly coarsely embed into a Hilbert space. 
\item $(G_n,S_n)_{n\in\mathbb N}$ does not weakly contain any expander.
\item Here we specify $A:=\Z$ with its standard generator, and $A_n:=\Z/2^n\Z$. Then, there exists a sequence of $1$-Lipschitz maps $f_n$ from $G_n$ to a Hilbert space such that for every sequence of subsets $\Sigma_n\subseteq G_n$ satisfying
\begin{equation}\label{eq:bounded}
\sup_{n\in \N}\sup_{k\in \Sigma_n}\inf_{g\in G_n}\|f_n(gk)-f_n(g)\|<\infty,
\end{equation}
there exist $i_0\in \N$ and $r>0$ such that $\Sigma_n$ lies in the $r$-neighborhood of $Y_n^{i_0}$.

\item The previous sequence $(f_n)_{n\in\mathbb N}$ is such that for every unbounded sequence of subgroups $(K_n)_{n\in\mathbb N}$, there exists a sequence $k_n\in K_n$ such that 
$$\inf_{g\in G_n}\|f_n(gk_n)-f_n(g)\|\to \infty$$ as $n\to\infty.$
\end{itemize}
\end{thm}
\begin{proof}
Let $Y:=\{v_g, \; g\in Q, v\in V\}.$
The first statement is a consequence of the fact that the pair $(G,Y)$ has relative Property T~\cite{CI}. 

The second statement is proved in the same way as the second statement of Theorem~\ref{thm:SL2}.

Let us turn to the proof of the third statement. For every $n\in \N$, consider the equivariant embedding $\phi_n\colon \Z/2^n\Z\to \C$, where every $k\in \Z/2^n\Z$ is sent to $2^ne^{2ik\pi/2^n}$. In other words, this corresponds to taking the orbit of $2^n\in \C$ under the action of $\Z/2^n\Z$ by rotations. Equip $\Z/2^n\Z$ with the word length associated to $\{\pm1\}$, 
that we denote by $|\cdot|_{\pm 1}$. That is, for $k,l\in \Z/2^n\Z$, we have $|k-l|_{\pm 1}=\min \{|k'-l'|,\; k'=k\!\!\mod 2^n, l'=l\!\!\mod 2^n\}.$
One checks that for all $k,l\in \Z/2^n\Z$,
$$|k-l|_{\pm 1}\leqslant |\phi_n(k)-\phi_n(l)|\leqslant 2\pi |k-l|_{\pm 1}.$$
Taking a cartesian product of these embeddings, we obtain an embedding $h_n:\;\bigoplus_{Q_n}\Z/2^n\Z \to \C^{Q_n}$. Note that for all $g,g'\in \bigoplus_{Q_n}\Z/2^n\Z$, we have
\begin{equation}\label{eq:comparison}
(d_{Y_n}(g,g'))^{1/2}\leqslant  \|h_n(g)-h_n(g')\|\leqslant  2\pi d_{Y_n}(g,g').
\end{equation}
Moreover, since $Y_n$ is invariant under conjugation in $G_n$, the action of $\bigoplus_{Q_n}\Z/2^n\Z$ by translation on its Cayley graph extends to all of $G_n$ by Lemma \ref{lem:extension}. Actually, one easily checks that the action of $G_n$ on $\bigoplus_{Q_n}\Z/2^n\Z$ extends to an action by isometry on $\C^{Q_n}$. To see this, it is enough to look at elements in a generating set of $G_n$. On the one hand, 
for every $q\in Q_n$, the element $\delta_q\in \bigoplus_{Q_n}\Z/2^n\Z$, which equals $1$ on $q$ and $0$ elsewhere, is clearly the restriction of an isometry of $\C^{Q_n}$ which acts by rotation on the $\C$-factor indexed by $q$ and trivially on the other factors. On the other hand, any element of $H_{8n-8}$ acts on the direct sum  $\bigoplus_{Q_n}\Z/2^n\Z$ by shifting the index, so it is also the restriction of an isometry of   $\C^{Q_n}$.
Hence, $h_n$ extends to an embedding of all of $G_n$. More precisely, this embedding is the orbit of the function $v_n\colon\; Q_n\to \C$ which is constant equal to $2^n$. It follows from (\ref{eq:comparison}) that it is $2\pi$-Lipschitz in restriction to $\bigoplus_{Q_n}\Z/2^n\Z$. Since elements of $H_{8n-8}$ act trivially on $v_n$, we deduce that this extension of $h_n$ is $2\pi$-Lipschitz on all of $(G_n,S_n)$. 

On the other hand, we consider the uniform coarse embedding $j_n$ of $H_{8n-8}$ from \cite{AGS} into a Hilbert space $\HH$. We obtain a Lipschitz map $f_n:=h_n\oplus j_n$ from $G_n$ to the orthogonal direct sum of these Hilbert spaces.

Now let $(\Sigma_n)_{n\in \N}$ be a sequence of subsets of $G_n$ satisfying (\ref{eq:bounded}).  Uniform properness of  the $j_n$'s implies that $\Sigma_n$ lies at uniformly bounded Hausdorff distance -- say $\leqslant  r$ -- from  a subset $\Sigma_n'\subseteq \bigoplus_{Q_n}\Z/2^n\Z$. But by (\ref{eq:comparison}), $\Sigma_n'$ must lie in a ball of radius $i_0$ for $d_{Y_n}$, where $i_0>0$ is independent of $n$. 

Finally, it is easy to adapt the previous argument to show the last statement (or directly deduce it from the previous one).
\end{proof}

\subsection{A construction which admits a fibered coarse embedding}\label{sec:fibered}

Observe that all of our examples above are box spaces of groups with relative Property T. It follows that these sequences do not admit any fibered coarse embedding into a Hilbert space. 
We give now an alternative construction of a box space of a group with the Haagerup property, which therefore does admit a fibered coarse embedding.

Let us proceed as in the preceding construction, replacing $\SL(3,\Z)$ by $\SL(2,\Z)$.  The resulting sequence of groups $G_n:=A_n\wr_{Q_n}H_{k_n}$ is now a box space of $G:=\Z\wr_{Q}\F_U$,
which has the Haagerup property by \cite[Theorem 1.5]{CSV}, as $Q$, being a subgroup of $\SL(2,\Z)$, has the Haagerup property. On the other hand, the sequence $(Q_n)_{n\in\mathbb{N}}$  satisfies a uniform spectral gap property by a famous result of Selberg. It follows from the proof of \cite{CI}*{Theorem 3.1} that the sequence $(G_n)_{n\in\mathbb{N}}$ has relative Property $T$ with respect to $Y$, in restriction to unitary representations which are direct sums of representations which factor through some $G_n$.
Observe that this restricted version of relative Property T is exactly what was needed to prove the Poincar\'e inequality (\ref {eq:PoincareSubset}) in Theorem~\ref{thm:wreath}. Therefore, it follows that the sequence $(G_n)_{n\in\mathbb{N}}$ does not coarsely embed into a Hilbert space.

\section{Open problems}\label{sec:questions} 

Our relative expanders, admitting no weakly embedded expanders, provide a new  type of geometric and analytic behavior among spaces of bounded geometry.
Therefore, we expect a further impact on coarse and metric geometry as well as on operator algebra. In particular, the following natural problems remain open.

\begin{itemize}

\item Do the examples constructed in Sections~\ref{sec:wreath} and \ref{sec:fibered} satisfy the (reduced) coarse Baum-Connes conjecture?\smallskip  

\item Construct a sequence of finite graphs with bounded degree, and unbounded girth which does not coarsely embed into a Hilbert space and yet has no weakly embeded expander (in particular: a box space of $\mathbb F_m$).\smallskip

\item Construct a finitely generated group which does not coarsely embed into a Hilbert space and yet has no weakly embedded expander. This is stated as an open problem in \cite{GK}, see also
\cite{NYu}*{Section 5.7}.\smallskip

\item Does there exist a graph with bounded degree which does not coarsely embed into a Hilbert space and yet coarsely embeds 
into an $L^p$-space for $p$ large enough? Is it possible to construct such an example using sequences of finite Cayley graphs? 
Note that it is not shown yet whether our second construction (based on generalized wreath products) provides a space which does not coarsely embed into any $L^p$-space for $p>2$.\smallskip 

\item More generally, it would be interesting to know in which metric spaces (Banach spaces, CAT(0)-spaces, etc.) relative expanders can -- or cannot -- be coarsely embedded.  In particular, the methods of \cite{Mi} and of \cite{MN1} might be used to extend the class of metric spaces (not only Banach spaces) in which relative expanders with respect to unbounded normal subgroups cannot coarsely embed. 
\smallskip 

\item It results from (\ref{trueexpander}), that expanders (or even large subsets of expanders) do not weakly embed into $\ell^p$ for all $1\leqslant p<\infty$. On the other hand all of our examples do weakly embed into $\ell^2$ (just take a coarse embedding of $Q_n$). Therefore, it seems relevant to ask whether a sequence of finite graphs with uniformly bounded degree that does not weakly embed into $\ell^2$ is necessarily an expander. 
\end{itemize}

\appendix
\section{Our relative expanders do not satisfy ``expander like" Poincar\'e inequalities}

The purpose of this appendix is to strengthen the fact that our relative expanders are very far from actual expanders. More precisely, we show that our box spaces do not satisfy Poincar\'e inequalities which are -- in a certain sense -- similar to those satisfied by actual expanders.

The main technical result of this section is the following observation.

\begin{prop}\label{prop:fiberwise}
Let $(G_n, S_n)_{n\in\mathbb N}$ be a sequence of semi-direct products $G_n:=N_n\rtimes Q_n$ together with subsets $S_n\subseteq G_n$ of fixed cardinality such that 
\begin{itemize}
\item $S_n$ generates $G_n$;

\item $N_n$ is finite for all $n$;

\item if $\pi_n\colon G_n\twoheadrightarrow Q_n$ denotes the canonical projection, the groups $(Q_n)_{n\in\mathbb N}$ are quotients of 
some exact\footnote{that is, $C^\ast_{\rm red}(Q)$ is an exact $C^\ast$-algebra or $Q$ has Guoliang Yu's Property A~\cite{NYu}.} group $Q$ such that the sequence of Cayley graphs $(Q_n,\pi_n(S_n))_{n\in\mathbb N}$ converges in the 
space of marked groups to some Cayley graph of $Q$;

\item the sequence $(N_n)_{n\in\mathbb N}$ equipped with the metric induced by $d_{S_n}$ coarsely embeds into a Hilbert space $\HH$.
\end{itemize}
Then there exists a sequence of 1-Lipschitz maps $\phi_n:(G_n,S_n)\to \HH$ whose restriction to $N_n$ is a coarse embedding of $(N_n)_{n\in\mathbb N}$ into $\HH$.
\end{prop}
The condition on $(Q_n)_{n\in\mathbb N}$ is satisfied if the sequence  $(Q_n,\pi_n(S_n))_{n\in\mathbb N}$ has girth tending to infinity as $n\to\infty$. 
Indeed, in this case, $(Q_n,\pi_n(S_n))_{n\in\mathbb N}$ converges in the space of marked groups to a Cayley graph of a free group. 
Thus, this proposition applies to all the box spaces we have constructed in \S \ref{sec:construc}. 

Before we prove the proposition, let us deduce our main motivation for this appendix, namely the following result (which also applies to the sequence of graphs we have constructed in \S \ref{sec:construc}). 

\begin{cor}\label{cor:last}
Let $(G_n, S_n)_{n\in\mathbb N}$ satisfy the hypothesis of Proposition \ref{prop:fiberwise} and assume, in addition, that the sequence $(Q_n)_{n\in\mathbb N}$ coarsely embeds into a Hilbert space $\HH$. Suppose that there exists a sequence of probability measures $\mu_n$ on $G_n$, supported on a sequence of subsets $A_n\subseteq G_n$ such that
\begin{itemize} 
\item $\mu_n$ is comparable to the uniform measure on $A_n$: $$\inf_{n,a\in A_n}|A_n|\mu_n(a)>0;$$

\item $G_n$ satisfies the following Poincar\'e inequality: there exists $C<\infty$ such that for all $n\in \N$ and every $1$-Lipschitz map $f_n\colon G_n\to \HH$,
\begin{equation}\label{eq:weakPoincare}
\sum_{a,b\in A_n}\|f_n(a)-f_n(b)\|^2\mu_n(a)\mu_n(b)\leqslant C.
\end{equation}
 
 \end{itemize}
Then $A_n$ is bounded:  there exists $R>0$ such that for all $n$, $A_n$ is contained in a ball of radius $R$ of $G_n$ with respect to $d_{S_n}$.
\end{cor}

\begin{proof}[Proof of Corollary~\ref{cor:last}]
The arguments are very similar to those employed in the proof of Proposition \ref{PropIntro:exactsequence}, therefore we only sketch them.
First assume that the cardinality of $A_n$ is bounded uniformly over $n\in \mathbb N$, then we see that $A_n$ must be bounded by applying  
(\ref{eq:weakPoincare}) to the $1$-Lipschitz function on $G_n$ defined by $g\mapsto f_{n,a}(g)=d_{S_n}(a,g)$
for some $a$ in $A_n$.
We assume now that the cardinality of $A_n$ is unbounded. Pick a $1$-Lipschitz coarse embedding $\psi_n$ of  $Q_n$ into $\HH$. Applying the Poincar\'e inequality to $\psi_n\circ \pi_n$ yields as in the proof of Proposition \ref{PropIntro:exactsequence} that there exists $c>0$ such that at least $c|A_n|$ points of $A_n$ are sent to a single point in $Q_n$.  In other words, a positive proportion of $A_n$ belongs to a single coset of $N_n$, say $N_n$ itself. But then applying   (\ref{eq:weakPoincare}) to the sequence $(\phi_n)_{n\in \mathbb N}$ of Proposition \ref{prop:fiberwise} yields the required contradiction.
\end{proof}

\begin{proof}[Proof of Proposition \ref{prop:fiberwise}]
Since $Q$ is exact, it follows from the proof of \cite{DG}*{Theorem 4.1} that the sequence of semi-direct products $N_n\rtimes Q$ admits a coarse embedding 
$\psi_n\colon N_n\rtimes Q\to \HH$. Now since 
$(Q_n)_{n\in \mathbb N}$ converges to $Q$ there exists a sequence $r_n\to\infty$ as $n\to\infty$, such that balls of radius $r_n$ in $Q_n$ are isometric to the ball of radius $r_n$ of $Q$. 
It follows that the subsets $[N_n]_{r_n}:=\{x\in G_n, \;d(x,N_n)\leqslant r_n\}$ of $G_n$ are isometric to the corresponding subsets of $N_n\rtimes Q$. In particular, $\psi_n$
can be seen as a coarse embedding of $[N_n]_{r_n}$ into $\HH$. 

First, up to ``slowing down" $\psi_n$ as explained at the beginning of the proof of Proposition \ref{PropIntro:exactsequence}, one can assume that $\psi_n$ is $1$-Lipschitz and satisfies that $\psi_n(N_n)$ is contained in the ball of radius $r_n$ centered at the origin of $\HH$.
 
Now we define $\phi_n\colon G_n\to\HH$ by $$\phi_n(g):=\psi_n(g)\left(1-\frac{d_{S_n}(g, N_n)}{r_n}\right)$$ if $g\in [N_n]_{r_n}$, and $0$ otherwise. One easily checks that in restriction to $[N_n]_{r_n}$, $\phi_n$ is $2$-Lipschitz: this is because $\psi_n$ is $1$-Lipschitz and satisfies $\|\psi_n(g)\|\leqslant r_n$, while $1-d_{S_n}(g, N_n)/r_n$ is $1/r_n$-Lipschitz and has absolute value at most $1$. Moreover, for $g$ at distance $r_n$ from $N_n$, $\phi_n(g)=0$. Therefore, we deduce from the triangle inequality, using that the distance in $G_n$ is geodesic, that $\frac12\phi_n$ is $1$-Lipschitz everywhere. 
\end{proof}


\bigskip
\footnotesize

\end{document}